\documentclass{amsart}

\usepackage{amsthm}
\usepackage[leqno]{amsmath}
\usepackage{latexsym,amsfonts,amssymb}
\usepackage[all]{xy} \SelectTips{eu}{} \SilentMatrices
\usepackage[draft]{hyperref}

\newcommand{\m}{\mathfrak{m}}

\newcommand{\p}{\mathfrak{p}}

\newcommand{\xra}[2][]{\xrightarrow[#1]{\;#2\;}}

\newcommand{\Rmk}{(R,\m,\k)}

\newcommand{\dimR}{\operatorname{dim}R}

\newcommand{\Spec}[1]{\operatorname{Spec}#1}

\newcommand{\ann}[2][R]{\operatorname{ann}_{#1}#2}

\newcommand{\dpt}[2][R]{\operatorname{depth}_{#1}#2}

\newcommand{\Hom}[3][R]{\operatorname{Hom}_{#1}(#2,#3)}

\newcommand{\Ext}[4][R]{\operatorname{Ext}_{#1}^{#2}(#3,#4)}

\makeatletter
\def\@nobreak@#1{\mathchoice%
  {\nobreakdef@\displaystyle\f@size{#1}}%
  {\nobreakdef@\nobreakstyle\tf@size{\firstchoice@false #1}}%
  {\nobreakdef@\nobreakstyle\sf@size{\firstchoice@false #1}}%
  {\nobreakdef@\nobreakstyle\ssf@size{\firstchoice@false #1}}%
  \check@mathfonts}%
\def\nobreakdef@#1#2#3{\hbox{{%
                    \everymath{#1}%
                    \let\f@size#2\selectfont%
                    #3}}}%
\makeatother

\renewcommand{\k}{\mathsf{k}}

\newcommand{\rank}{\operatorname{rank}}

\newcommand{\numberseries}{\mdseries}   

\newlength{\thmtopspace}                
\newlength{\thmbotspace}                
\newlength{\thmheadspace}               
\newlength{\thmindent}                  

\renewcommand{\subparagraph}{\vspace*{\thmbotspace}}

\newtheoremstyle{bfupright head,slanted body}
                {\thmtopspace}{\thmbotspace}
                {\slshape}{\thmindent}{\bfseries}{.}{\thmheadspace}
                {{\numberseries \thmnumber{(#2) }}\thmnote{#3}}

\newtheoremstyle{bfupright head,upright body}
                {\thmtopspace}{\thmbotspace}
                {\upshape}{\thmindent}{\bfseries}{.}{\thmheadspace}
                {{\numberseries \thmnumber{(#2) }}\thmnote{#3}}

\newtheoremstyle{bfit head,upright body}
                {\thmtopspace}{\thmbotspace}
                {\upshape}{\thmindent}{\upshape}{.}{\thmheadspace}
                {{\numberseries\thmnumber{(#2) }}
                {\bfseries\itshape\thmnote{\negthickspace#3}}}

\newtheoremstyle{it head,upright body}
                {\thmtopspace}{\thmbotspace}
                {\upshape}{\thmindent}{\upshape}{.}{\thmheadspace}
                {{\numberseries\thmnumber{(#2) }}
                {\itshape\thmnote{\negthickspace#3}}}

\newtheoremstyle{fixed bf head,slanted body}
                {\thmtopspace}{\thmbotspace}{\slshape}
                {\thmindent}{\bfseries}{.}{\thmheadspace}
                {{\numberseries \thmnumber{(#2) }}\thmname{#1}\thmnote{ (#3)}}

\newtheoremstyle{fixed bf head,upright body}
                {\thmtopspace}{\thmbotspace}{\upshape}
                {\thmindent}{\bfseries}{.}{\thmheadspace}
                {{\numberseries \thmnumber{(#2) }}\thmname{#1}\thmnote{ (#3)}}

\newtheoremstyle{fixed bfit head,upright body}
                {\thmtopspace}{\thmbotspace}{\upshape}
                {\thmindent}{\bfseries\itshape}{.}{\thmheadspace}
                {{\numberseries \thmnumber{(#2) }}\thmname{#1}\thmnote{ (#3)}}

\newtheoremstyle{sc head,small body}
                {\thmtopspace}{\thmbotspace}
                {\small\upshape}{\thmindent}{\scshape}{.}{\thmheadspace}
                {\thmname{#1}}

\newtheoremstyle{numbered paragraph}
                {\thmtopspace}{\thmbotspace}{\upshape}
                {\thmindent}{\upshape}{}{0pt}
                {{\numberseries \thmnumber{(#2) }}}

\newtheoremstyle{unnumbered paragraph}
                {\thmtopspace}{\thmbotspace}{\upshape}
                {\parindent}{\upshape}{}{0pt}

\theoremstyle{bfupright head,slanted body}
\newtheorem{res}{}[section]             \newtheorem*{res*}{}

\theoremstyle{bfit head,upright body}
                 \newtheorem*{com*}{}

\theoremstyle{bfupright head,upright body}
\newtheorem{bfhpg}[res]{}               \newtheorem*{bfhpg*}{}

\theoremstyle{it head,upright body}
               \newtheorem*{ithpg*}{}

\theoremstyle{sc head,small body}

\theoremstyle{fixed bf head,slanted body}
\newtheorem{thm}[res]{Theorem}          \newtheorem*{thm*}{Theorem}
\newtheorem{prp}[res]{Proposition}      \newtheorem*{prp*}{Proposition}
        \newtheorem*{cor*}{Corollary}
\newtheorem{lem}[res]{Lemma}            \newtheorem*{lem*}{Lemma}

\theoremstyle{fixed bf head,upright body}
       \newtheorem*{dfn*}{Definition}
     \newtheorem*{con*}{Construction}
      \newtheorem*{obs*}{Observation}
           \newtheorem*{rmk*}{Remark}
          \newtheorem*{exa*}{Example}
         \newtheorem*{exe*}{Exercise}
            \newtheorem{stp*}{Setup}

\theoremstyle{numbered paragraph}

\theoremstyle{unnumbered paragraph}
\newtheorem{ipg*}{}

\newlength{\thmlistleft}        
\newlength{\thmlistright}       
\newlength{\thmlistpartopsep}   
\newlength{\thmlisttopsep}      
\newlength{\thmlistparsep}      
\newlength{\thmlistitemsep}     

\newcounter{eqc} 
  {\end{list}}%

\newcounter{prt}
  {\end{list}}%

\newcounter{rqm}
  {\end{list}}%

  {\end{list}}%

  {\end{list}}%

\newcounter{exercise}
  {\end{list}}%

\newenvironment{prf*}[1][Proof]{%
  \begin{proof}[\bf #1]
    \setcounter{equation}{0}
    \renewcommand{\theequation}{\arabic{equation}}}
  {\end{proof}
}

\newcommand{\pgref}[1]{(\ref{#1})}

\renewcommand{\eqref}[1]{\pgref{eq:#1}}

\begin{document}

\sloppy \allowdisplaybreaks[4]

\title{Constructing Big Indecomposable modules}

\author[A. Crabbe]{Andrew Crabbe}
\email{s-acrabbe1@math.unl.edu}
\author[J. Striuli]{Janet Striuli} 
\thanks{ The second author was partially supported by NSF grant DMS~0201904}
\email{jstriuli2@math.unl.edu}
\address{Department of Mathematics, University of Nebraska, Lincoln,
  NE~68588, U.S.A.}

\date{\today}


\date{\today}
\begin{abstract}
  Let $R$ be local Noetherian ring of depth at least two.  We prove
  that there are indecomposable $R$-modules which are free on the
  punctured spectrum of constant, arbitrarily large, rank. 
\end{abstract}

\maketitle

\section{introduction}
A fruitful approach to the study of a commutative ring is to understand the
category of its finitely generated modules. Over zero dimensional rings,
it is feasible to understand the entire category by classifying its indecomposable objects; however, over larger rings, this classification is not a viable possibility.  A more tractable category that still encodes important information about the ring is that of maximal Cohen-Macaulay modules. For example, finiteness conditions on the category
of maximal Cohen-Macaulay modules detect certain types of singularities
\cite{Aus-85}, \cite{BGS-87}. Such finiteness conditions cannot necessarily be applied to other
categories of modules, as one can build arbitrarily {\it big}
indecomposable modules which are not maximal Cohen-Macaulay.

Our interest was raised by work of W. Hassler, R. Karr, L. Klinger and
R. Wiegand. Indeed, they build indecomposable modules of arbitrarily
large rank over hypersurface singularities with positive dimension
which are not $A_1$-singularities (\cite{HKKW-06}), and over rings
which are not homomorphic images of Dedekind-like rings
(\cite{HKKW-07}, \cite{HKKW-08}, \cite{HaWi-06}).  While over
hypersurfaces the indecomposable modules are free of constant rank on
the punctured spectrum, in the latter case, such modules are known to
be free after localizing only at a finite number of primes. Our main
result is the following theorem, which gives indecomposable modules
free, of arbitrarily large rank, on the punctured spectrum for any
ring of depth at least two.
\begin{thm}\label{mainIntro} 
  Let $R$ be a local Noetherian ring of depth at least two.  Given any
  integer $r$, there exists a short exact sequence of finitely
  generated $R$-modules:
  \[0 \to T \to X \to M \to 0,\] in which $T$ has finite length, $M$
  has positive depth, $X$ is indecomposable, and $M$ and $X$ are free
  of constant rank on the punctured spectrum, and such rank is at
  least $r$.  If further $R$ is Cohen-Macaulay, $M$ can be chosen
  maximal Cohen-Macaulay.
\end{thm}
For every $R$-module $N$, denote by $N^{(r)}$ the direct sum of $r$
copies of $N$.  When the ring is complete of dimension at least two
(with no restriction on the depth), we even prove that for every
integer $r$ and indecomposable module $M$, which is free of constant
rank on the punctured spectrum and has positive depth, there exists a
short exact sequence of $R$-modules $0\to T \to X \to M^{(r)} \to 0$,
where $X$ is indecomposable, and $T$ has finite length.  One cannot
relax the condition on the depth as it is known that the rank of
an indecomposable module is at most two certain some rings of
depth one, for example $A_1$-singularity (\cite{KL-05},
\cite{KL-01}, \cite{KL-01A}).

Theorem \ref{mainIntro} is particularly interesting when the ring is
Cohen-Macaulay of {\it bounded Cohen-Macaulay type}, by which we mean
that there is an upper bound on the minimal number of generators for
any indecomposable maximal Cohen-Macaulay module.  Theorem
\ref{mainIntro} says that one cannot possibly hope for such a bound on
indecomposable modules that are not maximal Cohen-Macaulay.

\section{Preliminaries}

In the following, $\Rmk$ is a local Noetherian ring and all modules are
finitely generated. Short exact sequences are identified with elements
of the module $\Ext{1}{-}{-}$; for details see \cite{MacL-75}.

 \begin{bfhpg}[Actions]\label{pullback}
   If $\alpha$ is an element $\Ext{1}{M}{N}$, and $f$
   is an $R$-homomorphism in $\Hom{M'}{M}$, then $\alpha f$ denotes
   the element in $\Ext{1}{M'}{N}$ obtained from the pullback of
   $\alpha$ via $f$, as shown in the following diagram
 \begin{equation*}
    \xymatrix{
      \alpha f \colon 0
      \ar[r]
      & N
        \ar[r]
        \ar@{=}[d]
        & Q
          \ar[r]
          \ar[d]
          & M'
            \ar[r]
            \ar[d]^-{f}
            & 0
      \\
      \alpha  \colon 0
      \ar[r]
      & N
        \ar[r]
        & X
          \ar[r]
          & M
            \ar[r]
            & 0. 
            }     
  \end{equation*}
  Set $\mathcal{B}=\Hom{M}{M}$, if $M=M'$ then the above action gives
  $\Ext{1}{M}{N}$ a structure of $\mathcal{B}$-module. To avoid any
  confusion we write $\Ext{1}{M}{N}_{\mathcal{B}}$.
\end{bfhpg}
In order to present our proof, we need to recall a result from
\cite{HaWi-06} which gives the tool to construct indecomposable
modules.

\begin{bfhpg}
  [Indecomposable modules]\label{roger}\cite[Theorem 2.3]{HaWi-06} Let
  $M$ be a finitely generated $R$-module of positive depth and let $T$
  be an indecomposable finitely generated $R$-module of finite
  length. Set $\mathcal{B}=\Hom{M}{M}$. Suppose there exists an element $\alpha \in \Ext{1}{M}{T}$ such that  $\ann[\mathcal{B}]{(\alpha)} \subseteq J(\mathcal{B})$,
  where $J(-)$ denotes the Jacobson radical.  If $0 \to T \to X\to M
  \to 0$ represents $\alpha$, then $X$ is
  indecomposable.
\end{bfhpg}
When the ring is complete, the following lemma gives us the element $\alpha$.

\begin{lem}\label{andrew2}
  Let $R$ be a complete ring and let $M$, $T$ be finitely generated
  $R$-modules.  Suppose that $M$ is indecomposable and set
  $\mathcal{A} = \Hom{M}{M}$. Let $g_1, \dots, g_r$ be part of minimal
  generating set for $\Ext{1}{M}{T}$ as a right
  $\mathcal{A}$-module. Set $\mathcal{C}=\mathsf{M}_{r\times
    r}(\mathcal{A})$, the ring of $r\times r$ matrices with entries in
  $\mathcal{A}$.  Define $\rho: \mathcal{C}_\mathcal{C} \to
  (\Ext{1}{M}{N}^{(r)})_{\mathcal{C}}$ to be the right
  $\mathcal{C}$-module homomorphism such that for every $\gamma\in
  \mathcal{C}$, $\rho(\gamma)=(g_1,....,g_r)\gamma$, via matrix
  multiplication. Then $\ker (\rho) \subseteq J(\mathcal{C})$.
  
  In particular there exists an element $\alpha \in \Ext{1}{M^{(r)}}{T}$,
  such that $\ann[\mathcal{B}]{(\alpha)}\subseteq J(\mathcal{B})$, where
  $\mathcal{B}:=\Hom{M^{(r)}}{M^{(r)}}$.
\end{lem}
\begin{proof}
  Let $\phi\in \ker (\rho)$ and write $\phi=(a_{ij})$, where $a_{ij}
  \in \mathcal{A}$.  For each $j=1,\dots ,r$, the equality $\sum
  _{i=1}^{r}g_ia_{ij}=0$ holds. Since $M$ is indecomposable, and $R$
  is complete, $\mathcal{A}$ is a local ring with the Jacobson radical
  $J(\mathcal{A})$ as its unique maximal two-sided ideal.  By the
  assumption on $g_1, \dots, g_r$, it follows that $a_{ij} \in
  J(\mathcal{A})$ for each $i$ and each $j$, \cite[Lemma
  4.43]{Rot-79}. Since $J(\mathcal{C}) = \mathsf{M}_{r \times
    r}(J(\mathcal{A}))$, the element $\phi$ belongs to
  $J(\mathcal{C})$; see for example \cite[Exercise 13, page
  433]{Hung-80}.

  For the in particular statement, there are isomorphisms $\phi:
  \Ext{1}{M^{(r)}}{N}\to \Ext{1}{M}{N}^{(r)}$ (see \ref{phi}) and
  $\psi: \mathcal{B} \to \mathcal{C}$ (see \ref{psi}) which are
  compatible with the action of $\mathcal{C}$ on $\Ext{1}{M}{N}^{(r)}$
  (matrix multiplication) and the action of $\mathcal{B}$ on
  $\Ext{1}{M^{(r)}}{N}$ (see \ref{pullback}), for details see
  Proposition \ref{comp}. The element $\alpha=\phi^{-1}((g_1, \dots,
  g_r))$ satisfies $\ann[\mathcal{B}] (\alpha)=\ker{\rho} \subseteq
  J(\mathcal{B})$.
\end{proof}

The following theorem is the first step in proving Theorem
\ref{mainIntro}. In the rest of the paper, given an $R$-module $M$, we
denote by $\nu_R(M)$, the minimal number of generators of $M$ as an
$R$-module.
\begin{thm}\label{main1}
  Let $\Rmk$ be a local complete Noetherian ring and  let
  $M$ be an indecomposable $R$-module of positive depth.
  Assume that 
  \begin{equation}\label{working}
     \limsup_{n\to \infty}\nu_{R}(\Ext{1}{M}{R/\m^n})=\infty.
  \end{equation}
  Let $r$ be an integer. Then
  there exists a short exact sequence
 \[
  \alpha \colon 0 \to R/\m^n \to X \to M^{(r)} \to 0,
 \] 
  such that  $X$ is an  indecomposable $R$-module.
\end{thm}

\begin{proof} Let $\mathcal{A}$ be the finitely generated $R$-algebra
  $\Hom{M}{M}$. Suppose that $\nu_R(\mathcal{A})=h$.  Then the following holds:
  \[
  \nu_{\mathcal{A}}(\Ext{1}{M}{R/\m^n})\geq
  \nu_R(\Ext{1}{M}{R/\m^n})/h, \; \text{and so}\]
  \[\limsup_{n\to
    \infty}\nu_{\mathcal{A}}(\Ext{1}{M}{R/\m^n})=\infty.\]
   
  For a given integer  $r$, choose $n$ such that
  $\nu_{\mathcal{A}}(\Ext{1}{M}{R/\m^n})>r$. Let $g_1, \dots , g_r$ be
  part of a minimal generating set of $ \Ext{1}{M}{R/\m^n}$ as an
  $\mathcal{A}$-module.  By Lemma \ref{andrew2}, there exists $\alpha \in
  \Ext{1}{M^{(r)}}{R/\m^n}$ such that $\ann[\mathcal{B}]{(\alpha)}\subseteq
  J(\mathcal{B})$, where $\mathcal{B}:=\Hom{M^{(r)}}{M^{(r)}}$. Let $\alpha$ be
  represented by the short exact sequence $0 \to R/\m^n \to X \to
  M^{(r)}\to 0$. By \ref{roger}, the $R$-module $X$ is indecomposable.
\end{proof}

The next step is to find modules for which  equation (\ref{working})
in Theorem \ref{main1} is satisfied.

\begin{bfhpg}[Generalized Hilbert polynomials]\label{GHP} Let $M$ be a
  finitely generated $R$-module. It is known that the lengths of the
  $R$-modules $\Ext{1}{M}{R/\m^nR}$ are given by a polynomial
  $\xi_{M,\m}(n)$ for all $n$ large enough (see \cite{KaTh-07},
  \cite{Kod-93}). In \cite{CKST-08}, the authors give an exact formula
  for the degree of this polynomial; in particular when
  $M$ has
  constant rank on the punctured spectrum, the degree is $\dimR -1$.
 \end{bfhpg}
Often if a module $M$ is free on the punctured spectrum then it has
constant rank.  This is the result of the following lemma which relies
on Hartshorne's Connectedness Theorem: if the depth of a local
Noetherian ring is at least two, then the punctured spectrum is
connected.
\begin{lem} \label{H} Let $R$ be a local Noetherian ring of
  depth at least two. Let $M$ be an $R$-module free in the
  punctured spectrum. Then $M$ is free of constant rank in the
  punctured spectrum.
\end{lem}
\begin{proof} For $i=1, \dots, n$, let $\mathfrak{p}_i$ be the minimal
  primes of $R$, and set
  $t=\max\{\rank_{R_{\mathfrak{p}_i}}(M_{\mathfrak{p}_i}) \mid i=1,
  \dots ,n\}$. We wish to show that $M$ is free of constant rank on
  the punctured spectrum and we will proceed by contradiction. Without
  loss of generality, we may assume that $\rank_{R_{\mathfrak{p}_i}}
  (M_{\mathfrak{p}_i})=t$ for all $i\leq s$ and
  $\rank_{R_{\mathfrak{p}_i}} (M_{\mathfrak{p}_i})<t$ for $s < i \leq
  n$. Set $I=\cap_{i=1}^{s}\mathfrak{p}_i$ and
  $J=\cap_{i=s+1}^n\mathfrak{p}_i$. For every ideal $L$, let
  $\mathcal{V}(L)$ be the closed set in $\Spec{R}$ which gives the
  support of $R/L$.  We claim that the two open sets
  $\Spec{R}\setminus \mathcal{V}(I)$ and $\Spec{R}\setminus
  \mathcal{V}(J)$ disconnect the punctured spectrum
  $\Spec{R}\setminus\{\m\}$, contradicting Hartshorne's Connectedness
  Theorem, see \cite{Eise-95}.  To prove the claim, it is enough to
  show that $\{\m\}=\mathcal{V}(I)\cap \mathcal{V}(J)$, since
  $\Spec{R}=\mathcal{V}(I)\cup\mathcal{V}(J)$ by the definition of $I$
  and $J$. Assume by contradiction that there exists a prime
  $\mathfrak{q}\neq \m$ in the set
  $\mathcal{V}(I)\cap\mathcal{V}(J)$. Then there exists an $i\leq s$
  and $j>s$ such that $\mathfrak{p}_i\subset \mathfrak{q}$ and
  $\mathfrak{p}_j\subset \mathfrak{q}$. But then
$
t=\rank_{R_{\mathfrak{p}_i}}
(M_{\mathfrak{p}_i})=\rank_{R_{\mathfrak{q}}}
(M_\mathfrak{q})=\rank_{R_{\mathfrak{p}_j}} (M_{\mathfrak{p}_j}) <t$.
\end{proof}
We end the section with a technical lemma.

\begin{lem}\label{janet1}
  Let $\{M_n\}$ be a family of finitely generated $R$-modules of
  finite length. Assume that for fixed integers $s > 0$, $a>0$, and
  $b$, the following hold: $\m^sM_n = 0$ and 
  $\lambda(M_n) \geq an + b$, for all $n\gg 0$. Then for any given
  integer $h$, there exists an $n$ such that $\nu_R(M_n)>h$.
\end{lem}
\begin{proof}
  Set $t=\max_{i=0, \dots, s}\{\nu_R(\m^i) \}$, let $h$ be a given
  integer and choose $k>th$.  Choose $n$ large enough so that $(a_1n
  + a_0)/s \geq k$. There is an equality on lengths
\[
 \lambda(M_n) = \lambda(M_n/\m
 M_n) + \lambda(\m M_n/\m^2 M_n) + \cdots + \lambda(\m^{s-2}
 M_n/\m^{s-1} M_n) + \lambda(\m^{s-1} M_n),
\]
 and since $\lambda(M_n) \geq a_1n + a_0$,
 there must be an $a < s$ for which $ \lambda(\m^{a} M_n/\m^{a+1}
 M_n)\geq (a_1n + a_0)/s$. This implies  
\begin{align*}
t\nu_R(M_n)  & \geq \nu(\m^a)\nu(M_n) \geq \nu(\m^aM_n) \\
            &= \lambda(\m^{a} M_n/\m^{a+1} M_n)\\
            &\geq (a_1n + a_0)/s \geq k > th,
\end{align*}
 and therefore $\nu_R(M_n)>h$.
 \end{proof}

\section{Main Results}
In this section we prove Theorem \ref{mainIntro} from the
introduction. We first handle the case where the ring is complete. If
$M$ is an $R$-module, we denote by $\Omega^i_R(M)$ the $i$th syzygy of
$M$ in a minimal free resolution.
\begin{thm}\label{main2}
  Let $\Rmk$ be a local complete ring of dimension at least two. Let
  $M$ be an indecomposable $R$-module of positive depth that is free
  of constant rank on the punctured spectrum. Given any integer $r$,
  there exists an integer $n$ and a short exact sequence
  \[ 0 \to R/\m^n \to X \to M^{(r)} \to 0, \] where $X$ is
  indecomposable, free on the punctured spectrum with
  $\rank_{R_{\p}}(X_{\p}) \geq r$, for every prime ideal $\p \neq \m$.
\end{thm}

\begin{proof}
  Since $M$ is free on the punctured spectrum,
  $\Ext{1}{M}{\Omega^1_R(M)}$ is a module of finite length, so
  there exists an integer $s$ such that
  $\m^s\Ext{1}{M}{\Omega^1_R(M)}=0$. Let $\beta$ be the element $0\to
  \Omega^1_R(M) \to F \to M \to 0$ in $\Ext{1}{M}{\Omega^1_R(M)}$, and let
  $N$ be an $R$-module. For every element $\alpha \in \Ext{1}{M}{N}$
  there exists a homomorphism $f \in \Hom{\Omega^1_R(M)}{N}$ such that
  $\alpha=\beta f$.  Therefore $\m^s\subseteq \ann (\Ext{1}{M}{N})$
  for every finitely generated module $N$, and in particular the equality 
  $\m^s\Ext{1}{M}{R/\m^n}=0$ holds for every integer $n>0$.  By
  hypothesis on the dimension of the ring and by \ref{GHP}, the
  lengths of the modules $\Ext{1}{M}{R/\m^n}$ are given by a
  polynomial of degree one.  By Lemma \ref{janet1}, $\limsup_{n\to
    \infty}\nu_R(\Ext{1}{M}{R/\m^n})=\infty$, and by Theorem
  \ref{main1}, for every positive  integer $r$ there exists a short exact sequence
\[ 
0\to R/\m^n \to X \to M^{(r)} \to 0,
\]
where $X$ is indecomposable. For every prime ideal $\p\in
\Spec(R)$, the module $X_{\p}$ is isomorphic to
$M_{\p}^{(r)}$, which is free by hyphotesis.
\end{proof}

We now remove the assumption on $R$ being complete, and prove Theorem
\ref{mainIntro} from the introduction. Given a finitely generated
$R$-module $P$, denote by $H^0_{\m}(P)$ the largest submodule of
finite length of $P$.

\begin{proof} Let $d=\dimR$ and let $N$ be an indecomposable summand
  of the $R$-module $\Omega^d_R(\k)/H^0_{\m}(\Omega^d_R(\k))$. Write $N\otimes
  \widehat{R} =\oplus _{i=1}^kN_i$, where the $N_i$ are indecomposable
  summands of
  $\Omega^d_{\widehat{R}}(\k)/H^0_{\m}(\Omega^d_{\widehat{R}}(\k))$.  The modules $N_i$ have positive depth and are free
  on the punctured spectrum, thus by Lemma \ref{H} and by the
  hypothesis on the depth of the ring, they have constant rank.  Fix an
  integer $n$. For each $i=1, \dots ,k$, by Theorem \ref{main2}, there
  exists an integer $n_i$ and a short exact sequence
\[
\alpha_i :  \; 0 \to R/\m^{n_i} \to Y_i \to N_i^{(n)} \to 0,
\]
 where $Y_i$ is an indecomposable $\widehat{R}$-module. 
Consider the sequence 
\[
\oplus_{i=1}^{k}\alpha_i: \; 0 \to \oplus_{i}^{k}R/\m^{n_i}\to
\oplus_{i}^{k}Y_i \to \oplus_{i=1}^k N_i^{(n)} \to 0.
\]
The right-hand and left-hand modules are extended\footnote{We say that
  an $\widehat{R}$-module $M$ is extended if there exists an
  $R$-module $N$ such that $N \otimes \widehat{R}\cong M$}. Since the
length of $\Ext{1}{N^{(n)}}{\oplus _{i=1}^{k}R/\m^{n_i}}$ is finite, by
\cite[Proposition 3.2(1)]{FWW-07} the middle module is also an
extended module, so say $\oplus_{i=1}^k Y_i \cong Y\otimes
\widehat{R}$, where $Y$ is a finitely generated $R$-module. In
particular, there exists a short exact sequence:
\[
\xymatrix{
0 \ar[r]&\oplus_{i=1}^{k} R/\m^{n_i} \ar[r]^-{\delta} &Y \ar[r]& N^{(n)} \ar[r]& 0,}
\]
where $Y$ is an $R$-module free on the punctured spectrum of constant
rank equal to $n \cdot \rank(N)$.  If $Y$ is indecomposable, then we
are done.  If not, since $Y$ has at most $k$ indecomposable summands, we may choose an indecomposable summand $X$ of $Y$ that is free on the punctured spectrum with rank at least $n \cdot \frac{\rank(N)}{k}$. Write $Y=X \oplus Z$. Let $T$ be $\delta ^{-1}(X)$, and $L=\delta^{-1}(Z)$
and so $\oplus_{i=1}^{k} R/\m^{n_i}= T \oplus L$. Set $M=
X/\delta(T)$ and $P=Z/\delta(L)$.  We have that the following diagram
commutes:
\[
\xymatrix{
  0 \ar[r]&L  \ar[r]^-{\delta |_L} & Z  \ar[r] &P \ar[r] &0\\
  0\ar[r]&  \oplus_{i=1}^{k} R/\m^{n_i}\ar[u]^{\pi}\ar[r]^-{\delta}& Y=X \oplus Z \ar[u]^{\pi}\ar[r] & N^{(n)}\ar[u]^-{f} \ar[r] &0\\
  0 \ar[r] & T \ar[u]^-{\iota}\ar[r]^{\delta|_T}& X
  \ar[r]\ar[u]^-{\iota} & M\ar[r]\ar[u]^-{g}&0,}
\]
where $\iota$ is the inclusion and $\pi$ is the projection. By the
Snake Lemma, $f$ is surjective and $g$ is injective, and therefore
we have a short exact sequence $0 \to M \to N^{(n)} \to P \to 0$.  By
symmetry, there exists a short exact sequence $0 \to P
\to N^{(n)} \to M \to 0$. By the below lemma, $\dpt{M}=\dpt{N}$. Since
$\dpt{N}>0$, the short exact sequence $0 \to T \to X \to M \to 0$ in
the diagram satisfies the conclusion of the theorem. 

If the ring $R$ is Cohen-Macaulay, then $N$ is a direct summand of
$\Omega^d_{R}(\k)$, which is a maximal Cohen-Macaulay
module. 
\end{proof}

\begin{lem}Let $R$ be a local Noetherian ring and assume there are
  two short exact sequences $0\to M \to N \to P \to 0$ and $0 \to P
  \to N \to M \to 0$. Then $\dpt{M}=\dpt{P}=\dpt{N}$.
\end{lem}
\begin{proof} It is enough to prove that $\dpt{M}=\dpt{N}$. Recall
  that for every finitely generated $R$-module $L$,
  $\dpt{L}=\min\{i\mid \Ext{i}{\k}{L}\neq 0 \}$ (see for example
  \cite{bruns-herzog}). If $\dpt{M}<\dpt{N}$, then, by applying the
  functor $\Hom{\k}{-}$ to the short exact sequence $0 \to M \to N \to
  P \to 0$, $\dpt{P}< \dpt{M}$, and hence $\dpt{P}< \dpt{N}$. By
  applying the functor $\Hom{\k}{-}$ to the sequence $0\to P\to N \to M \to
  0$, we obtain $\dpt{M}<\dpt{P}$, which is a contradiction.
\end{proof}


\section{Appendix}
In the following we record some known facts about the $R$-module
$\Ext{1}{M}{N}$, where $M$ and $N$ are two finitely generated
$R$-modules. Our main source is \cite{MacL-75}.

Let $N$ and $M_i$, for $i=1, \dots n$, be finitely generated
$R$-modules. In the following, $\mathsf{m}$ denotes an element in $
\oplus_{i=1}^{n}M_i$, and $\mathsf{m}_j \in M_j$ is the $j$th
component of $\mathsf{ m}$. Let $\iota _j: M_j \to \oplus_{i=1}^{n}M_i$ be
the inclusion such that if $\iota_j(m)=\mathsf{ m}$ then 
\[
\mathsf{ m}_i=\begin{cases}m \;\text{for}\; i=j,\\
                          0 \;\text{ for}\; i\neq j.\end{cases}
\] 
Let $\pi_j:
\oplus_{i=1}^{n}M_i\to M_j$ be the projection: $\pi_j(\mathsf{m})=\mathsf{m}_j$.
  \begin{bfhpg}[Isomorphism of modules]\label{phi} There exists an
    $R$-isomorphism
  \[
  \phi: \Ext{1}{\oplus_{i=1}^{n} M_i}{N}  \to \oplus_{i=1}^{n} \Ext{1}{M_i}{N},
  \]
  which maps $\alpha$ to $[\alpha \iota _1, \dots ,\alpha \iota_n]$,
  where the right action of $\iota_j$ on $\alpha$ is described in \ref{pullback}
  (see \cite[page 71]{MacL-75}).  Vice versa, if $\alpha_i$ is
  represented by the short exact sequence
  \[
  0 \to N \xra{\iota_i} X_{\alpha_i} \xra{\rho_i} M_i \to 0
  \]
  then $\phi^{-1} ([\alpha _1, \ldots ,\alpha _n])$ is represented as
  the bottom short exact sequence, where $\nabla_N$ is the addition of
  $n$ elements from $N$, and $P$ is the pushout:
  \begin{equation*}
    \xymatrix{
      0
      \ar[r]
      & N^{(n)}
        \ar[r]^-{\oplus \iota_i}
        \ar[d]^-{\nabla _N}
        & \oplus _{i=1}^{n} X_{\alpha_i}
          \ar[r]^{\oplus \rho_i}
          \ar[d]
          & \oplus _{i=1}^{n} M_i
            \ar[r]
            \ar@{=}[d]
            & 0
      \\
       0
      \ar[r]
      & N
        \ar[r]
        & P
          \ar[r]
          & \oplus _{i=1}^{n} M_i
            \ar[r]
            & 0. 
            }     
  \end{equation*}
  \end{bfhpg}
  \begin{bfhpg}[Isomorphism of algebras]
    \label{psi} Let $\mathcal{B}=\Hom{\oplus M_i}{\oplus M_i}$ and let
    $\mathcal{C}$ be the $R$-algebra of $n\times n$ matrices
    $B=(b_{ij})$ where $b_{ij} \in \Hom{M_i}{M_j}$. We have the
    following isomorphism:
\[
\psi: \, \mathcal{B} \to \mathcal{C}, \; \psi(b)=B^T, 
\]
  where $B=(b_{ij})$ and $b_{ij}= \pi_j\circ b \circ \iota_i$.
  \end{bfhpg} 
  On the one hand, $\Ext{1}{\oplus M_i}{N}$ is a right
  $\mathcal{B}$-module, as described in (\ref{pullback}). On the other
  hand, $\mathcal{C}$ acts on $ \oplus_{i=1}^{n} \Ext{1}{M_i}{N}$: let
  $[\alpha_1,\ldots, \alpha_n]$ be a row vector in $\oplus_{i=1}^{n}
  \Ext{1}{M_i}{N}$, and $C=(b_{ij})$ be a matrix in $\mathcal{C}$,
  then $[\alpha_1,\ldots\alpha_n]C$ is given by matrix multiplication,
  where the right action of $b_{ij}$ on $\alpha _i$ is given as in
  (\ref{pullback}).  In the following, it is important to remember
  that a homomorphism acts on the right on short exact sequences, and
  on the left on elements. For example if $f, g \in \Hom{M}{M}$ for
  some module $M$, $m\in M$, and $\alpha \in \Ext{1}{M}{N}$, then we
  write $\alpha fg$ for $(\alpha f)g$ but $g\circ f (m)$ for
  $g(f(m))$.

  \begin{prp}\label{comp}
    Let $\phi$ and $\psi$ be the isomorphisms from \ref{phi} and
    \ref{psi}. The equality $ \phi(\alpha b)=\phi(\alpha)\psi(b)$
    holds for every $\alpha \in \Ext{1}{\oplus_{i=1}^{n}M_i}{N}$ and
    $b \in \mathcal{B}$.
 \end{prp}
 \begin{proof}
   To see this, it is enough to show that equality holds for an element
   $b \in \mathcal{B}$ such that all but at most one entry $b_{ij}$ is
   zero, since the map $\phi$ is an $R$-homomorphism and therefore is
   additive.
 
   Fix two integers $k$ and $l$. Fix an element $b\in \mathcal{B}$ such that
   $\psi(b)=B^T$, where $B=(b_{ij})$, and $b_{ij} \neq 0$ if and only if $i=k$ and
   $j=l$. Let $\alpha \in \Ext{1}{\oplus_{i=1}^{n}M_i}{N}$ be represented
   by the short exact sequence:
 \begin{equation*}
 \xymatrix{
    \alpha \colon 0
      \ar[r]
      & N
        \ar[r]
        & X
          \ar[r]^-{g}
          & \oplus _{i=1}^{n} M_i
            \ar[r]
            & 0. 
            }     
  \end{equation*}
  We first compute the left-hand side of the equality \ref{comp};
  $\alpha b$ is given by
 \begin{equation*}
    \xymatrix{
      \alpha b \colon 0
      \ar[r]
      & N
        \ar[r]
        \ar@{=}[d]
        & X_b
          \ar[r]^-{g_b}
          \ar[d]
          & \oplus _{i=1}^{n} M_i
            \ar[r]
            \ar[d]^-{b}
            & 0
      \\
      \alpha \colon 0
      \ar[r]
      & N
        \ar[r]
        & X
          \ar[r]^-{g}
          & \oplus _{i=1}^{n} M_i
            \ar[r]
            & 0, 
            }     
  \end{equation*}
 where $X_b\subset X \oplus (\oplus _{i=1}^{n} M_i)$, is given by  
\[
 X_b=\{(x,\mathsf{m})\mid g(x)=b(\mathsf{m})=[0,\dots, b_{kl} (\mathsf{m}_k),0 \dots ,0]\},
\]
 and $g_b((x,\mathsf{m}))=\mathsf{m}$. 
 
 Now $\phi(\alpha b)=[\alpha b \iota_1, \ldots, \alpha b\iota_n]$, where 
  \begin{equation*}
    \xymatrix{
      \alpha b\iota_j \colon 0
      \ar[r]
      & N
        \ar[r]
        & Y_j
          \ar[r]^-{h_j}
          & M_j
            \ar[r]
            & 0
          }
     \end{equation*}
     where $Y_j \subset X_b\oplus M_j \subset X \oplus (\oplus
     _{i=1}^{n} M_i)\oplus M_j$ is given by
   \[
   Y_j=\{(x,\mathsf{m},m) \mid g(x)=[0,\dots,b_{kl}(\mathsf{m}_k),0 \dots
   ,0],\; \iota_j(m) = g_b((x,\mathsf{m}))=\mathsf{m}\},
   \]
    and $h_j((x,\mathsf{ m},m_j))=m$.
     It follows that
     \begin{enumerate}
     \item $Y_j\cong N \oplus M_j$ for $j \neq k$, and therefore $
       \alpha f\iota_j=0$;
     \item  $Y_k\cong\{(x,m_k)\in X\oplus M_k \mid g(x)=[0,\dots ,b_{kl}(m),
     \dots, 0]\}$.
    \end{enumerate} 
    In particular, $\phi(\alpha b)=[0,\dots,(\alpha
    b)\iota_k,\dots,0]$.

     To compute the right side of the equality, we first compute
     $\phi(\alpha)=[\alpha \iota _1, \dots ,\alpha \iota_n]$, where
     $\alpha\iota _j$ is given as
\begin{equation*}
    \xymatrix{
      \alpha \iota _j \colon 0
      \ar[r]
      & N
        \ar[r]
        \ar@{=}[d]
        & X_j
          \ar[r]^-{g_j}
          \ar[d]
          & M_j
            \ar[r]
            \ar[d]^-{\iota_j}
            & 0
      \\
      \alpha \colon 0
      \ar[r]
      & N
        \ar[r]
        & X
          \ar[r]^-{g}
          & \oplus _{i=1}^{n} M_i
            \ar[r]
            & 0, 
            }     
\end{equation*}
where $X_j=\{(x,m_j)\in X\oplus M_j \mid g(x)=\iota_j(m_j)\}$ and $g_j((x,m_j))=m_j$.
It follows that $\phi(\alpha)\psi(b)=\phi(\alpha)B^T=(0,\dots, (\alpha
\iota_l)b_{kl}, \dots ,0)$, where the $k$-th entry is the only
non-zero one.
 
All that is left to see is that $(\alpha \iota_l)b_{kl}=(\alpha
b)\iota_k$. But there is an equality of homomorphisms
$\iota_lb_{kl}=b\iota_k: M_k \to \oplus_{i=1}^{n}M_i$, so equality
holds.
\end{proof}


\section*{Acknowledgments}
We thank Roger Wiegand and Lars W. Christensen for their 
suggestions.

\bibliographystyle{amsplain}


\begin{thebibliography}{10}

\bibitem {Aus-85}
    Maurice Auslander, \emph {Finite type implies isolated singularity},
  Orders and their applications, Lecture Notes in Math.
   \textbf {1142}, Springer, Berlin, 1985. \MR{812487}

\bibitem {BGS-87}
  Buchweitz, R.-O. and Greuel, G.-M. and Schreyer, F.-O.,
    \emph{Cohen-{M}acaulay modules on hypersurface singularities. {II}},
   Invent. Math., \textbf{88},(1987), no.~1, \MR{14B05 (13C15)}

\bibitem{bruns-herzog}
Winfried Bruns and J{\"u}rgen Herzog, \emph{Cohen-{M}acaulay rings}, Cambridge
  Studies in Advanced Mathematics, vol.~39, Cambridge University Press,
  Cambridge, 1993. \MR{1251956}

\bibitem{Eise-95}
David Eisenbud, \emph{Commutative algebra with a view toward algebraic
  geometry}, Graduate Texts in Mathematics, vol. 150, Springer-Verlag, New
  York, 1995. \MR{97a:13001}

\bibitem{FWW-07}
Anders~J. Frankild, Sean Sather-Wagstaff, and Roger Wiegand, \emph{Ascent of
  module structures, vanishing of ext, and extended modules}, arXiv:0707.4197v2.

\bibitem{HKKW-06}
W.~Hassler, R.~Karr, L.~Klingler, and R.~Wiegand, \emph{Large indecomposable
  modules over local rings}, J. Algebra \textbf{303} (2006), no.~1, 202--215.
  \MR{2253659}

\bibitem{HKKW-07}
Wolfgang Hassler, Ryan Karr, Lee Klingler, and Roger Wiegand, \emph{Big
  indecomposable modules and direct-sum relations}, Illinois J. Math.
  \textbf{51} (2007), no.~1, 99--122 (electronic). \MR{2346189}

\bibitem{HKKW-08}
\bysame, \emph{Indecomposable modules of large rank over {C}ohen-{M}acaulay
  local rings}, Trans. Amer. Math. Soc. \textbf{360} (2008), no.~3, 1391--1406
  (electronic). \MR{2357700}

\bibitem{HaWi-06}
Wolfgang Hassler and Roger Wiegand, \emph{Big indecomposable mixed modules over
  hypersurface singularities}, Abelian groups, rings, modules, and homological
  algebra, Lect. Notes Pure Appl. Math., vol. 249, Chapman \& Hall/CRC, Boca
  Raton, FL, 2006, pp.~159--174. \MR{2229110 (2007h:13018)}

\bibitem{Hung-80}
Thomas~W. Hungerford, \emph{Algebra}, Graduate Texts in Mathematics, vol.~73,
  Springer-Verlag, New York, 1980, Reprint of the 1974 original. \MR{MR600654}

\bibitem{CKST-08}
Daniel Katz, Andrew Crabbe, Janet Striuli, and Emanoil Theodorescu,
  \emph{Hilbert polynomials for the controvarient functor}, preprint (2007).

\bibitem{KaTh-07}
Daniel Katz and Emanoil Theodorescu, \emph{On the degree of hilbert polynomials
  associated to the torsion functor}, Proc. Amer. Math. Soc. \textbf{135}
  (2007), 3073--3082.

\bibitem{KL-01A}
Lee Klingler and Lawrence~S. Levy, \emph{Representation type of commutative
  {N}oetherian rings. {I}. {L}ocal wildness}, Pacific J. Math. \textbf{200}
  (2001), no.~2, 345--386. \MR{1868696}

\bibitem{KL-01}
\bysame, \emph{Representation type of commutative {N}oetherian rings. {II}.
  {L}ocal tameness}, Pacific J. Math. \textbf{200} (2001), no.~2, 387--483.
  \MR{1868697}

\bibitem{KL-05}
\bysame, \emph{Representation type of commutative {N}oetherian rings. {III}.
  {G}lobal wildness and tameness}, Mem. Amer. Math. Soc. \textbf{176} (2005),
  no.~832, viii+170. \MR{2147090}

\bibitem{Kod-93}
Vijay Kodiyalam, \emph{Homological invariants of powers of an ideal}, Proc.
  Amer. Math. Soc. \textbf{118} (1993), no.~3, 757--764. \MR{1156471}

\bibitem{MacL-75}
Saunders Mac~Lane, \emph{Homology}, Classics in Mathematics, Springer-Verlag,
  Berlin, 1995, Reprint of the 1975 edition. \MR{96d:18001}

\bibitem{Rot-79}
Joseph~J. Rotman, \emph{An introduction to homological algebra}, Pure and
  Applied Mathematics, vol.~85, Academic Press Inc. [Harcourt Brace Jovanovich
  Publishers], New York, 1979. \MR{538169}


\end{thebibliography}
\end{document}